\documentclass[11pt]{article}
\usepackage{epsfig}
\usepackage{tikz}
\usepackage{caption}
\usepackage{subcaption}
\usepackage{graphicx}
\usepackage{hyperref}
\usepackage{pdfpages}
\usepackage{amsthm,amssymb,amsmath}
\usepackage{amsfonts}
\usepackage{float}
\usepackage{roundbox}
\usepackage{url}
\usepackage{color,soul}
\usepackage{algorithm}
\usepackage{algorithmic}



\usepackage{epsfig}
\usepackage{cleveref}
\usepackage{wrapfig}
\usepackage{lipsum}
\usepackage{setspace,xargs}
\usepackage{colortbl}
\usepackage{threeparttable}
\usepackage{hhline}
\usepackage{pst-grad} 
\usepackage{pst-plot} 
\usepackage[space]{grffile}
\usepackage{etoolbox} 
\usepackage[perpage]{footmisc} 
\usepackage{fancyhdr}
\usepackage{cite}
\usepackage[nottoc]{tocbibind}
\usepackage[textfont=it]{caption}
\usepackage{titletoc}
\usepackage{xcolor}
\usepackage{pstricks,pst-node,pst-tree}
\usetikzlibrary{decorations.pathreplacing}
\usetikzlibrary{arrows,positioning} 
\usepackage{titlesec}
\usetikzlibrary{positioning,calc,matrix}    
\usetikzlibrary{decorations.markings}
\hypersetup{colorlinks = true,linkcolor={blue},citecolor={blue},urlcolor={blue}}
\input{epsf}
\newtheorem{thm}{Theorem}

\newtheorem{lem}{Lemma}
\newtheorem{con}{Conjecture}

\begin{document}
\title{
{\Large{\bf Star edge coloring of Cactus graphs}}}
{\small
\author{
{Behnaz Omoomi}\footnote{Research is partially supported by the Iran National Science Foundation (INSF)}, {Marzieh Vahid Dastjerdi}, and {Yasaman Yektaeian}\\
[1mm]
{\small \it  Department of Mathematical Sciences}\\
{\small \it  Isfahan University of Technology} \\
{\small \it 84156-83111, \ Isfahan, Iran}}

\date{}
\maketitle \baselineskip15truept
\begin{abstract}
\noindent A star edge coloring of a graph $G$ is a proper edge coloring of $G$ such that no path or cycle of length four is bi-colored. The star chromatic index of $G$, denoted by $\chi^{\prime}_{s}(G)$,  is the minimum $k$  such that $G$ admits a star edge coloring with $k$ colors. 
Bezegov{\'a}~et~al. (Star edge coloring of some classes of graphs, J.  Graph Theory, 81(1), pp.73-82. 2016) conjectured that the star chromatic index of outerplanar graphs with maximum degree $\Delta$,  is at most $\left\lfloor\frac{3\Delta}{2}\right\rfloor+1$.  In this paper,  we prove this conjecture for a class of outerplanar graphs, namely Cactus graphs, wherein every edge belongs to at most one cycle.\\


\end{abstract}

\section{Introduction}
All graphs in this paper are simple and undirected. Let $G$ be a graph with vertex set $V(G)$ and edge set $E(G)$. We show every edge $e \in E(G)$ with two endpoints  $u,v \in V(G)$, by $e=uv$, and say that  $u$ and $v$ are adjacent.
   Moreover, we say that two edges $e$ and $e^{\prime}$ are \textit{adjacent}  if they have a common endpoint.

A \textit{proper edge coloring} of $G$ is an assignment of colors to its edges  such that  no  two adjacent  edges   have the same color.
 There are variants  coloring of graphs under the additional  constraints. 
 For example,  
\textit{star edge coloring}  is a kind of  proper edge coloring with no  bi-colored path or cycle of length four (path or cycle with four edges).
 The \textit{star chromatic index}  of $G$, denoted by $\chi^{\prime}_{s}(G)$, is the minimum $k$ such that $G$ has a star edge coloring with $k$ colors \cite{Dvorak,Liu}. Star edge coloring was defined in 2008 by Liu and Deng \cite{Liu}.

In \cite{Bezegova}, Bezegov{\'a} et al. obtained upper bound $\left\lfloor\frac{3\Delta}{2}\right\rfloor$ for the star chromatic index of trees. Omoomi et~al. in  \cite{Omoomi},  gave a polynomial time algorithm to find the star chromatic index of every tree. Using the star chromatic index of trees,  Bezegov{\'a} et al. found upper bound $\left\lfloor\frac{3\Delta}{2}\right\rfloor+12$ for the star chromatic index of outerplnar graphs  \cite{Bezegova}.  An \textit{outerplanar} graph  is a graph that has a planar drawing  for which all vertices belong to the outer face.  Bezegov{\'a} et al. also presented the following conjecture.
\begin{con}\rm{\cite{Bezegova}}\label{con1}
If $G$ is an outerplanar graph with maximum degree $\Delta$, then
 \[\chi^\prime_s(G)\leq \left\lfloor\frac{3\Delta}{2}\right\rfloor+1.\]
\end{con}

In  \cite{wang}, Wang et al. proved that 
$\chi^{\prime}_s(G) \leq \left\lfloor\frac{3\Delta}{2}\right\rfloor+5$ for  outerplanar graph $G$ with maximum degree $\Delta$.

A {\it Cactus}  is a graph in which every edge belongs to at most one cycle. Since these graphs are outerplanar,  in order to prove Conjecture~\ref{con1}, it is worth to study the star edge coloring of Cactus graphs.
 In this paper,   we 
prove Conjecture~\ref{con1} for Cactus graphs with maximum degree~$\Delta$.

A {\it unicyclic Cactus}  (or UCC  for short) is a Cactus $G=C\cup F$, where $C$ is a cycle  (or an edge) and $F$ is a forest consists of some  rooted trees  with height  at most two and the roots in $C$ (see Figure~\ref{fig:1}). The {\it height} of a rooted tree is the length of the longest path between the root and a leaf.  We denote a rooted tree with root $v$, by $T_v$.
 
Through the paper, we assume that the degree of all vertices except leaves are $\Delta$,  because adding some leaves to the vertices with   degree at least two, does not reduce the star chromatic index.  We call such a graph a \textit{$\Delta$-semiregular}. 

The structure of the paper is as follows. In Section \ref{EO-coloring}, we present an algorithm to give a star edge coloring for every $\Delta$-semiregular UCC with at most $\left\lfloor\frac{3\Delta}{2}\right\rfloor+1$ colors.  Using this coloring,   in Section~\ref{Star-edge-coloring of a Cactus Graph}, we find a star edge coloring for every $\Delta$-semiregular Cactus with at most $\left\lfloor\frac{3\Delta}{2}\right\rfloor+1$ colors. Finally, in Section~\ref{tightness}, we show the tightness of the obtained  bound   for  an infinite family of Cactus graphs. 
\section{Star edge coloring of  unicyclic Cactus graphs}\label{EO-coloring}
We first introduce the terminology and notations that we need through the paper. 
 For further information on graph theory concepts and terminology we refer the reader to \cite{bondy}.

Let $G$ be a graph. We say $G$ is 2-{\it connected} if   between every two vertices  there are two internally disjoint paths.
A {\it block} in $G$ is a maximal 2-connected subgraph in $G$. 
Thus, every block in a Cactus is either an edge or  a cycle.
A {\it block graph} $H$ of  $G$ is a graph
that its vertices are blocks of $G$ and two vertices of $H$ are adjacent if and only if their corresponding blocks in $G$ intersect in a vertex.
Clearly, every block graph is a tree.

For every block $C$  in Cactus  $G$, we consider three types of edge, as follows.
\begin{align*}
& E_{1}(C)=\{e : e \in E(C) \}.
\\
&E_{2}(C,x) = \{e =xy: x \in V(C), y \not\in V(C) \}
, 
 ~~E_2(C) = \bigcup_{x \in V(C)} E_2(C,x).\\
& E_{3}(C,e)= \{e^\prime : e^\prime \not\in E_1(C)\cup E_2(C), e^\prime ~\text{is adjacent to}~e, e\in E_2(C) \}
,
~~
E_3(C)=\displaystyle \bigcup_{e \in E_2(C)}{E_3(C,e)}.
\end{align*}
   In Figure \ref{fig:1}, a UCC graph with its three type edges is shown.
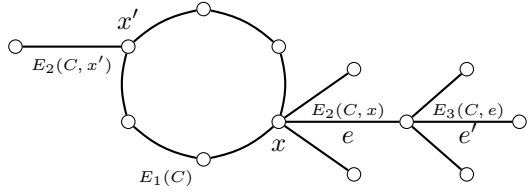
\begin{figure}[H]
\centering
\scalebox{1}{
\begin{tikzpicture}
[n/.style={circle,draw=black,scale=0.9,inner sep=2pt},
m/.style={draw=black}
]
\node[n] (1) at (0,.5)[] {};
\node[n] (2) at (1,1)[label=below: \small{$x$}] {};
\node[n] (3) at (1,2) []{};
\node[n] (44) at (0,2.5)[] {};
\node[n] (4) at (-1,2) [label=above: \small{$x^{\prime}$}]{};
\node[n] (5) at (-1,1)[] {};
\node[n] (2b) at (2.7,1) {};
\node[n] (22b) at (2,1.7) {$  $};
\node[n] (222b) at (2,0.3) {$  $};
\node[n] (2bb) at (4.2,1) {$ $};
\node[n] (2bbb) at (3.5,1.7) {$ $};
\node[n] (2bbbb) at (3.5,0.3) {$ $};
\node[n] (4b) at (-2.5,2) {$ $};

\draw[thick](1)to[bend right=15](2)node[midway,above,yshift=0cm,xshift=-.5cm]{\tiny{$E_1(C)$}};;
\draw[thick](2)to[bend right=15](3);
\draw[thick](3)to[bend right=15](44)node[midway,above,yshift=-50pt,xshift=0pt]{};
\draw[thick](4)to[bend left=15](44);
\draw[thick](4)to[bend right=15](5);
\draw[thick](5)to[bend right=15](1);
\draw[thick](2)--(22b);
\draw[thick](2)--(222b);
\draw[thick](2)--(2b)node[midway,above,yshift=-3pt,xshift=2pt]{\tiny{$E_2(C,x)$}}node[midway,above,yshift=-12pt,xshift=2pt]{\small{$e$}};
\draw[thick](2b)--(2bbbb);
\draw[thick](2b)--(2bbb);
\draw[thick](2b)--(2bb)node[midway,above,yshift=-3pt,xshift=2pt]{\tiny{$E_3(C,e)$}}node[midway,above,yshift=-12pt,xshift=2pt]{\small{$e^{\prime}$}};
\draw[thick](4)--(4b)node[midway,above,yshift=-.5cm,xshift=0cm]{\tiny{$E_2(C,x^{\prime})$}};
\end{tikzpicture}
}
\vspace*{-15mm}
\caption{A UCC graph  with three types of edges.}\label{fig:1}
\end{figure}
In every edge coloring $c$ of $G$ with $k$ colors, we show the color set by $\mathcal{C}:=\{1,2,\ldots, k \}$, the color set of all edges incident to  vertex $x$ by $C(x)$, and  the color of  edge $e$ by $c(e)$.
We now present an algorithm to give  a star edge coloring of every $\Delta$-semiregular UCC, with $\left\lfloor\frac{3\Delta}{2}\right\rfloor+1$ colors.
 
\begin{thm}\label{th:EO-coloring-sun}
If $G$ is a  UCC graph with maximum degree $\Delta$, then
 \[\chi^{\prime}_s(G)\leq \left\lfloor\frac{3\Delta}{2} \right\rfloor+1.\]
\end{thm}
\begin{proof} 
Without loss of generality, assume that $G=C\cup F$ is  $\Delta$-semiregular.
 In Algorithm~\ref{Coloring of UCC graph}, we present a  star edge coloring of $G$ with $\left\lfloor\frac{3\Delta}{2} \right\rfloor+1$ colors.
\begin{algorithm}[h]
{ \fontsize{9}{10} \selectfont
\caption{Star edge coloring of $\Delta$-semiregular UCC  Graphs}
\label{Coloring of UCC graph}
\begin{algorithmic}[1]
\REQUIRE  Graph $G=C\cup F$ as a $\Delta$-semiregular UCC, color set $\mathcal{C}=\{1,\ldots,\left\lfloor\frac{3\Delta}{2}\right\rfloor+1\}$.
\vspace{0.1cm}
\ENSURE Star edge coloring of  $G$ with color set $\mathcal{C}$.
\vspace{0.2cm}
\STATE Give an optimum star edge coloring of  $C$
\STATE Set $\mathcal{C}^\prime=\mathcal{C} \setminus \cup_{e\in E_1(C)}{c(e)}$.
\STATE  Color edges of $E_2(C)$ with $\Delta -2$  colors from $\mathcal{C}^\prime$.
\FOR{each $x$  in $V(C)$}
    \FOR{$i$ from 1 to $\Delta-2$}
         \STATE Set $e_{i}$ the $i$-th  edge  incident  to $x$ in $T_{x}$.
          \STATE   Color $\left \lfloor \frac{\Delta}{2} \right\rfloor+1$ edges of $E_3(C,e_{i})$ with colors of $\mathcal{C}\setminus C(x)$.
           \IF {$i  \ is \ even$}
    		    \STATE   Use colors of edges $e_{2k+1},e_{2k^{\prime}}$, where $2k+1<i$ and $i<2k^{\prime}$, for uncolored edges in $E_3(C,e_{i})$.
          \ELSE
    		    \STATE Use colors of edges $e_{2k},e_{2k^{\prime}+1}$, where $2k< i$ and  $i <2k^\prime+1$, for uncolored edges in $E_3(C,e_{i})$.
	     \ENDIF
     \ENDFOR
\ENDFOR
\end{algorithmic}
}
\end{algorithm}

The performance of Algorithm~\ref{Coloring of UCC graph} is as follows. 
We first  give an optimum star edge coloring for  block $C$ in $G=C\cup F$, where $\chi^{\prime}_s(C ) \leq 4$  (see the proof of Theorem 5.1 in \cite{Dvorak}). 
Let $\mathcal{C}^\prime$   be the set of colors that are not used for coloring $C$. 
We choose $\Delta-2$ colors from $\mathcal{C}^\prime$   for coloring the edges of $E_2(C)$. 
Assume that, $x\in V(C)$ and $e_{i}\in T_{x}$ is the $i$-th edge incident  to $x$, where  $1\leq i\leq \Delta-2$.

We color $\left \lfloor \frac{\Delta}{2} \right\rfloor+1$ edges of $E_3(C,e_{i})$ with colors that are not used in $C(x)$. 
To complete the colors of the edges in $E_3(C,e_{i})$ we have two possibilities, index $i$ is even or odd.
 If $i$ is even (or odd), we use colors of incident edges to $x$ with even indices more (or less) than $i$ and odd indices less (or more) than $i$.  Note that in both cases, there exist at least $\lfloor\frac{\Delta}{2}\rfloor-1$ colors  for the uncolored  edges in $E_3(C,e_{i})$.  Thus, we can color all  edges in $G$ with at most $\left\lfloor\frac{3\Delta}{2}\right\rfloor+1$ colors.
 
 We now prove that  this coloring is a star edge coloring. For this  purpose, we first  check the coloring of paths of length four in $T_{x}$. 
 Consider two arbitrary  edges  $e_{j}=xy_j$ and $e_{k}=xy_k$ in $E_2(C,x)$. Without loss of generality assume that  $j<k$.
   If the parity of $j$ and $k$ are different, then $c(e_{k})\not\in C(y_j)$.
Similarly, if the parity of $j$ and $k$ is the same, then  $c(e_{j})\not\in C(y_k)$. 
 Thus, we have no bi-colored path of length four in $T_{x}$.
 Moreover, since  the color set of the edges in $E_3(C)$ is disjoint from the color set of  $C$,  the obtained coloring is a star edge coloring of $G$.
\end{proof}
\section{Star edge coloring of Cactus graphs}\label{Star-edge-coloring of a Cactus Graph}
In this section, we prove Conjecture \ref{con1} for  Cactus graphs.
\begin{thm}\label{thm:main-result1}
If $G$ is a  Cactus with maximum degree $\Delta$, then
 \[\chi^{\prime}_s(G) \leq \left\lfloor \frac{3\Delta}{2}\right\rfloor+1.\]
\end{thm}
\begin{proof}
In \cite{Bezegova}, it is proved that the star chromatic index of outerplanar graphs with maximum degree three is at most  $5$. Thus, we prove the statement for   $\Delta\geq 4$.

 Let  $\sigma=(C_1,\ldots,C_t)$ be an enumeration   of blocks in $G$, in the order in which they are visited by {\it  breath first search} (BFS) in block graph $G$. The BFS is a traversing algorithm where  we start traversing from a selected vertex and explore all of the neighbour vertices at the present level prior to moving on to the vertices at the next level. 
 
 For every block $C_i$, $1\leq i\leq t$,  we construct a UCC graph  $G_{C_i}$, corresponds to three types edges of $C_i$ in $G$, as follows.   If there exists a cycle $D$ of length 3 or 4 in $E_2(C_i)\cup E_3(C_i)$,  then   remove edges of $E(D) \cap E_3(C_i)$ and add one new edge to each common vertex between deleted edges and the edges in $E_2(C_i)$  (see Figure~\ref{fig:X}).

 We now present the following algorithm, to provide a star edge coloring for every $\Delta$-semiregular Cactus with $\left\lfloor \frac{3\Delta}{2}\right\rfloor+1$ colors.
\begin{algorithm}[H]
{ \fontsize{9}{10} \selectfont
\caption{Star edge coloring of Cactus graphs}\label{Al2}
\label{Coloring of cactus}
\begin{algorithmic}[1]
\REQUIRE  $\Delta$-semiregular Cactus $G$, enumeration $\sigma=(C_1,\ldots,C_t)$ for the blocks of $G$ obtained by BFS in block graph $G$,
  color set $\mathcal{C}=\{1,\ldots,\left\lfloor\frac{3\Delta}{2}\right\rfloor+1\}$.
\vspace{0.1cm}
\ENSURE Star edge coloring of  $G$ with color set $\mathcal{C}$.
\vspace{0.2cm}
\FOR{ $i$ from 1 to $t$}  
      \STATE Construct UCC graph $G_{C_i}$ from three types edges of $C_i$ in $G$.
      \IF{$i=1$}\label{alg:line1-col1}
          \STATE Apply Algorithm~\ref{Coloring of UCC graph} for  $G_{C_1}$.\label{alg:linealg1}
      \ELSE
      \STATE  Consider restriction partial coloring of $G$ for their corresponding edges in $G_{C_i}$.
           \FOR{Every vertex $x$ of $C_i$}
                  \IF{Edges of $T_x$ is uncolored}
                       \STATE  Run lines 2,3, and 5 to 13 in Algorithm~\ref{Coloring of UCC graph} for $C_i\cup T_x$.
                  \ELSIF{Eedges of $E_3(C_i)\cap T_x$ are uncolored}
                        \STATE Run lines 5 to 13 in Algorithm~\ref{Coloring of UCC graph} for $C_i\cup T_x$.
                  \ENDIF
           \ENDFOR
       \ENDIF \label{alg:linel-col1}
           \STATE Set $c_i$ as the obtained coloring of $G_{C_i}$.
           \STATE Color uncolored edges of  $E_1(C_i)\cup E_2(C_i)$ in $G$ as same as the corresponding edges in $G_{C_i}$.
       \FOR{ every uncolored cycle $D_x\neq C_i$ with vertex $x$ in $C_i$ }\label{alg:line1-D}
       \STATE Set $e_1$ and $e_2$ as the edges incident to $x$ in  $D_x$, where $E_3(C_i,e_1)$ is colored before $E_3(C_i,e_2)$ in~$G_{C_i}$.
    \STATE Enumerate edges of $D_x$ as $D_x=e_1,e_2,\ldots,e_n$, where $n$ is the length of $D$.
       \STATE Set $C^\prime(x)=\mathcal{C}\setminus C(x)$.
   \IF{ $n=0 \pmod{3}$}
                  \STATE Select edges $e\in E_3(C_i,e_1)$ and $e^\prime\in E_3(C_i,e_2)$ of $G_{C_i}$ with the same color $c^\prime$ in $C^\prime(x)$.
                  \STATE Complete the coloring of $D_x$ in $G$ by coloring pattern $\underbrace{c_i(e_1),c_i(e_2),c^\prime},\ldots$.
             \ELSIF{$n=1\pmod{3}$}
                   \STATE Select edge  $e^{\prime}\in E_3(C_i,e_{2})$ in $G_{C_i}$, with  $c_i(e^\prime)\in$ $C^\prime(x)$. 
                 		 \IF{$n=4$ and $\Delta\neq 4$}
                				 \STATE Select edge  $e\in E_3(C_i,e_1)$,  with $ c_i(e)\in C(x)$.
                   			     \STATE Complete the coloring of $D_x$ using pattern $c_i(e_1),c_i(e_2),c_i(e^\prime),c_i(e)$.
              		      \ELSE
                  			      \STATE Complete the coloring of $D_x$ using coloring pattern $c_i(e_1),\underbrace{c_i(e_2),c_i(e_1),c_i(e^\prime)},\ldots$.
                			  \ENDIF
            \ELSE
                     \STATE   Select edges $e\in E_3(C_i,e_1)$ and $e^{\prime}\in E_3(C_i,e_{2})$ in $G_{C_i}$, with  $c_i(e)\in C^\prime(x)$ and $c_i(e^\prime)=c_i(e_1)$.
                      \IF{ $n=5$}
                        \STATE Choose color $c^\prime\in C^{\prime}(x)\setminus c_i(e_1) $.
                        \STATE Complete  the coloring of $D_x$ using coloring pattern $c_i(e_1),c_i(e_2),c_i(e_1),c^\prime,c_i(e)$.
                        \ELSE
                           \STATE  Complete the coloring of $D_x$ using coloring pattern\\ $c_i(e_1),c_i(e_2),\underbrace{c_i(e_1),c_i(e),c_i(e_2)},\ldots,\underbrace{c_i(e_1),c_i(e),c_i(e_2)},c_i(e_1),c_i(e_2),c_i(e)$.
                          \ENDIF                      
            \ENDIF  
      \ENDFOR \label{alg:linel-D}
      \STATE Color the remaining uncolored edges of  $E_3(C_i)$ in $G$ as  same as the coloring of corresponding edges  in $G_{C_i}$.      
\ENDFOR
\end{algorithmic}
}
\end{algorithm}
 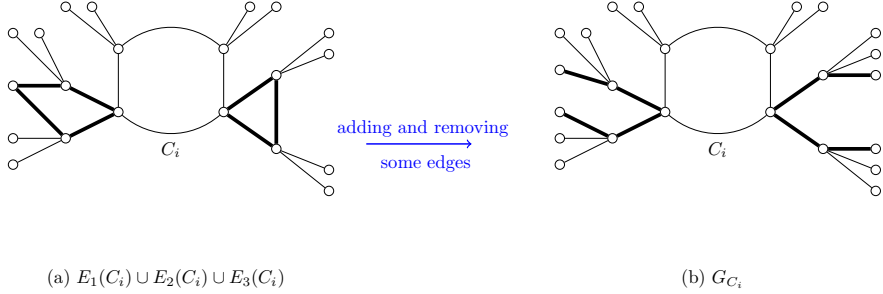
\begin{figure}[h]
\centering
\scalebox{0.7}{
\begin{subfigure}{0.4\textwidth}
\begin{tikzpicture}
[n/.style={circle,draw=black,scale=0.9,inner sep=2pt},
m/.style={draw=black}
]
\node[n] (2) at (1,1) { };
\node[n] (w) at (1,2.2) { };
\node[n] (z) at (-1,2.2) { };
\node[n] (5) at (-1,1) { };
\node[n] (22b) at (2,1.7) { };
\node[n] (222b) at (2,0.3) { };
\node[n] (c1) at (3,2.5) {$  $};
\node[n] (c2) at (3,2.1) {$  $};
\node[n] (d1) at (3,-.1) {$  $};
\node[n] (d2) at (3,-.5) {$  $};
\node[n] (y1) at (-2,1.5) { };
\node[n] (y2) at (-2,0.5) { };
\node[n] (5bbb-bb) at (-3,0) {};
\node[n] (5bbb-bbb) at (-3,0.5) {};
\node[n] (x) at (-3,1.5) {};
\node[n] (z1) at (-2,3) {$ $};
\node[n] (z2) at (-1.5,3) {$ $};
\node[n] (w1) at (1.5,3) {$ $};
\node[n] (w2) at (2,3) {$ $};
\node[n] (y11) at (-3,2.5) {$ $};
\node[n] (y12) at (-2.5,2.5) {$ $};

\draw(2)--(w);
\draw(w)to[bend right=40](z)node[midway,above,yshift=-50pt,xshift=0pt]{};
\draw(5)to[bend right=40](2)node[midway,above,yshift=0pt,xshift=0pt]{\small{$C_i$}};
\draw(z)--(5);
\draw[line width=2pt](2)--(22b);
\draw[line width=2pt](2)--(222b);
\draw[line width=2pt](22b)--(222b);
\draw(22b)--(c1);
\draw(22b)--(c2);
\draw(222b)--(d1);
\draw(222b)--(d2);
\draw(z)--(z1);
\draw(z)--(z2);
\draw(w)--(w1);
\draw(w)--(w2);
\draw[line width=2pt](5)--(y1);
\draw[line width=2pt](5)--(y2);
\draw[line width=2pt](y2)--(x);
\draw(y2)--(5bbb-bb);
\draw(y2)--(5bbb-bbb);
\draw[line width=2pt](x)--(y1);
\draw(y1)--(y11);
\draw(y1)--(y12);
\end{tikzpicture}
\caption{$E_1(C_i)\cup E_2(C_i)\cup E_3(C_i)$}\label{fig:X.1}
\end{subfigure}
\begin{subfigure}{0.1\textwidth}
\centering
\begin{tikzpicture}
\draw[thick, blue,->](0,0)--(2,0)node[midway,above,yshift=0pt,xshift=2pt]{\small{adding and removing}} node[midway,below,yshift=-2pt,xshift=2pt]{\small{ some edges}};
\end{tikzpicture}
\end{subfigure}
\hspace*{2.5cm}
\centering
\begin{subfigure}{0.4\textwidth}
\begin{tikzpicture}
[n/.style={circle,draw=black,scale=0.9,inner sep=2pt},
m/.style={draw=black}
]
\node[n] (2) at (1,1) {};
\node[n] (w) at (1,2.2) {};
\node[n] (z) at (-1,2.2) { };
\node[n] (5) at (-1,1) { };
\node[n] (22b) at (2,1.7) {$  $};
\node[n] (222b) at (2,0.3) {$  $};
\node[n] (c1) at (3,2.5) {};
\node[n] (c2) at (3,2.1) {};
\node[n] (d1) at (3,-.1) {};
\node[n] (d2) at (3,-.5) {};
\node[n] (y1) at (-2,1.5) { };
\node[n] (y2) at (-2,0.5) { };
\node[n] (5bbb-bb) at (-3,0) {};
\node[n] (5bbb-bbb) at (-3,0.5) {};
\node[n] (z1) at (-2,3) {$ $};
\node[n] (z2) at (-1.5,3) {$ $};
\node[n] (w1) at (1.5,3) {$ $};
\node[n] (w2) at (2,3) {$ $};
\node[n] (y11) at (-3,2.5) {$ $};
\node[n] (y12) at (-2.5,2.5) {$ $};
\node[n] (t1) at (-3,1.8) {};
\node[n] (t2) at (-3,1) {};
\node[n] (n1) at (3,1.7) {$ $};
\node[n] (n2) at (3,.3) {$ $};

\draw(2)--(w);
\draw(w)to[bend right=40](z)node[midway,above,yshift=-50pt,xshift=0pt]{};
\draw(5)to[bend right=40](2)node[midway,above,yshift=0pt,xshift=0pt]{\small{$C_i$}};
\draw(z)--(5);
\draw[line width=2pt](2)--(22b);
\draw[line width=2pt](2)--(222b);
\draw(22b)--(c1);
\draw(22b)--(c2);
\draw[line width=2pt](22b)--(n1);
\draw(222b)--(d1);
\draw(222b)--(d2);
\draw[line width=2pt](222b)--(n2);
\draw(z)--(z1);
\draw(z)--(z2);
\draw(w)--(w1);
\draw(w)--(w2);
\draw[line width=2pt](5)--(y1);
\draw[line width=2pt](5)--(y2);
\draw(y2)--(5bbb-bb);
\draw(y2)--(5bbb-bbb);
\draw[line width=2pt](y1)--(t1);
\draw[line width=2pt](y2)--(t2);
\draw(y1)--(y11);
\draw(y1)--(y12);
\end{tikzpicture}
\caption{$G_{C_i}$}\label{fig:X.2}
\end{subfigure}
}
\caption{Construction of UCC graph $G_{C_i}$.}\label{fig:X}
\end{figure}

Algorithm~\ref{Al2} runs as follows.  In $i$-th round, we consider $i$-th block  of $G$ in enumaration $\sigma$,
  and construct graph $G_{C_i}$. 
Note that for $i>1$,   $G_{C_i}$ has a partial star edge coloring that edges of $C_i$ and some edges of $E_2(C_i)\cup E_3(C_i)$ are colored.
 The edges in the same level of $T_x$ either are already colored or not.
In lines~\ref{alg:line1-col1}~to~\ref{alg:linel-col1} of Algorithm~\ref{Al2},  we complete the coloring of $G_{C_i}$.
%

   We now color  edges of  $E_1(C_i)\cup E_2(C_i)$ in $G$ as same as their correspondig edges  in $G_{C_i}$. 
    Let $H_{C_i}$ be the induced subgraph of $G$ on $E_1(C_i)\cup E_2(C_i)\cup E_3(C_i)$ and  all cycles  share a  vertex with $C_i$.  We obtain a coloring of  $H_{C_i}$, as follows.
    
     We first complete coloring of every uncolored cycle that  share a vertex with $C_i$ in $G$, according to their  length (see lines~\ref{alg:line1-D}~to~\ref{alg:linel-D} in the algorithm). 
 For example, let $D_x=e_1,e_2,\ldots,e_n$ be an uncolored cycle with two edges $e_1$ and $e_2$  incident to $x\in C_i$. For $D_x$ with different lengths,
we demonstrate its coloring in   Figure~\ref{fig2}~and~\ref{fig3},  by the following assumptions. 
Let $C(x)=\{1,\ldots,\Delta\}$,  $c(e_1)=1$, $c(e_{2})=2$, $\{A,B\}\subseteq C^\prime(x)$, $A\neq B$,  $\lambda\in (C(x_1)\cap C(x))\setminus\{1,2\}$, and dashed edges are in $G_{C_i}\setminus G$.
     
      Finally, we color the remaining edges of $E_3(C_i)$ in $H_{C_i}$ as  same as the coloring of corresponding edges in  $G_{C_i}$.
  Note that, in lines~\ref{alg:line1-col1}~to~\ref{alg:linel-col1}, edges of $E_3(C_i,e_1)$ is colored before $E_3(C_i,e_2)$. Thus, $c_i(e_2)$  is not used for the edges in $E_3(C_i,e_1)$ and every colors of $C^\prime(x)$ is used in $E_3(C_i,e_1)$ and $E_3(C_i,e_2)$. Hence, it is easy to check that the coloring of $C_i\cup T_x\cup D$ is a star edge coloring.
  
We now, by induction on $t$,  prove that the obtained coloring is a star edge coloring of $G$. 
For $i=1$, by applying Algorithm~\ref{Coloring of UCC graph} in line~\ref{alg:linealg1} of Algorithm~\ref{Al2}, the statment is obvious.
 Now, assume that  after $(i-1)$-th round in Algorithm~\ref{Al2} there is no bi-colored path of length 4, but  bi-colored path $P:=e_1,e_2,e_3,e_4$ appears after $i$-th round.
We have two ossibilities: all edges of $P$ belong to $H_{C_i}$ or not. 
In the first case, since for each vertex $x$ of $C_i$ there is no bi-colored path in $C_i\cup T_x\cup D_x$,  two adjacent edges of $P$  belong to $C_i$ or are in the same $T_x$. Therefore, $e_2$ or $e_3$ is in $C_i$ and is colored before  $i$-th round. 
Without loss of generality, assume that $e_3\in C_i$.
If $e_1$ or $e_2$ is uncolored   before $i$-th round, then for some $x$ in $C_i$,  $e_1\in T_x$ and by coloring $T_x$ in  $i$-th round, colors of $e_1$ and $e_3$ are different,  that is a contradiction.  Thus, edges $e_1$ and $e_2$ is also colored before $i$-th round. 

Now, let $P$ has some edges out of $H_{C_i}$. In this case, because of enumaration $\sigma$,   $P$ has three colored edges before $i$-th round. 
According to given argument,  this conclude that $e_1e_2e_3$ is a bi-colored path obtained before $i$-th round.
If $e_2\not\in C_i$, then all paths of length four that contain $P$ are colored before $i$-th round  and by the induction hypothesis, $P$ is not bi-colored. Thus, $e_2\in C_i$ and is  colored before $i$-th round. Since $e_4\in E_2(C_i)\cup E_3(C_i)$ and is colored in  $i$-th round  by Algorithm~\ref{Coloring of UCC graph},  color of $e_4$ is not the same as color of $e_2$. Therefore, $P$ is not bi-colored, that again is a contradiction.
Hence, the obtained coloring in Algorithm~\ref{Coloring of cactus} is a star edge coloring of $G$. 
\end{proof}
\begin{figure}[H]
\centering
\scalebox{0.8}{
\begin{subfigure}{0.33\textwidth}
\centering
\begin{tikzpicture}
[n/.style={circle,draw=black,scale=0.7,inner sep=2pt},
nn/.style={circle,,scale=0.7,inner sep=2pt, draw=black, densely dotted},
m/.style={draw=black}
]
\node[n] (1) at (-1,0.4){};
\node[n] (2) at (1,0.4){}; 
\node[n] (3) at (0,-0.5){$x$};
\node[n] (5) at (-1.1,-1.5){$x_1$};
\node[n] (4) at (1.1,-1.5){$x_2$};
\node[nn] (555) at (-1.5,-2.5){};
\node[nn] (44) at (0.8,-2.5){};
\node[n] (444) at (1.5,-2.5){};

\draw[m](1) to[bend left=50]node[midway,yshift=7pt,xshift=-1pt] {$C_i$} (2);
\draw[m](3) to[]node[midway,yshift=7pt,xshift=-1pt] {\footnotesize{$2$}}(4);
\draw[m](3) to[]node[midway,yshift=7pt,xshift=0pt]{\footnotesize{$1$}} (5);
\draw[m](1)--(3)node[midway,above,yshift=-2pt,xshift=2pt] {\footnotesize{$ $}};
\draw[m](2)--(3)node[midway,above,yshift=-3pt,xshift=0pt] {\footnotesize{$ $}};
\draw[m](4)to[bend left=10](5)node[midway,above,yshift=-61pt,xshift=0pt]
{$A$};
\draw[m,dashed](4)to[bend right]  node[midway,left,yshift=-3pt,xshift=12pt] {\footnotesize{$A$}} (44);
\draw[m, dashed](5) to[bend right] node[midway,right,yshift=-3pt,xshift=-15pt] {\footnotesize{$A$}} (555);
\draw[m](4) to[bend left]node[midway,yshift=0pt,xshift=5pt] {\footnotesize{1}}  (444);
\end{tikzpicture}
\caption{}\label{fig:1.1}
\end{subfigure}}
\scalebox{0.8}{
\begin{subfigure}{0.33\textwidth}
\centering
\begin{tikzpicture}
[n/.style={circle,draw=black,scale=0.7,inner sep=2pt},
nn/.style={circle,,scale=0.7,inner sep=2pt, draw=black, densely dotted},
m/.style={draw=black}
]
\node[n] (1) at (-1,0.5){};
\node[n] (2) at (1,0.5){}; 
\node[n] (3) at (0,-0.5){$x$};
\node[n] (5) at (-1,-1.5){$x_1$};
\node[n] (4) at (1,-1.5){$x_2$};
\node[n] (6) at (0,-2){};
\node[n] (7) at (2.5,-1.5){};
\node[n] (555) at (-2,-2.5){};
\node[nn] (55) at (-1,-2.5){};
\node[nn] (44) at (1,-2.5){};
\node[n] (444) at (2,-2.5){};

\draw[m](1) to[bend left=50]node[midway,yshift=7pt,xshift=-1pt] {$C_i$} (2);
\draw[m](3) to[]node[midway,yshift=7pt,xshift=-1pt] {\footnotesize{$2$}}(4);
\draw[m](3) to[]node[midway,yshift=7pt,xshift=0pt]{\footnotesize{$1$}} (5);
\draw[m](3) to[bend left]node[midway,yshift=5pt,xshift=0pt]{\footnotesize{$\lambda$}} (7);
\draw[m](1)--(3)node[midway,above,yshift=0pt,xshift=0pt] {\footnotesize{$ $}};
\draw[m](2)--(3)node[midway,above,yshift=-1pt,xshift=0pt] {\footnotesize{$ $}};
\draw[m](6)--(5)node[midway,below,yshift=2pt,xshift=-1pt] {\footnotesize{$\lambda$}};
\draw[m](6)--(4)node[midway,below,yshift=2pt,xshift=1pt] {\footnotesize{$A$}};
\draw[m,dashed](4)to[]  node[midway,right,yshift=-1pt,xshift=-2pt] {\footnotesize{$A$}} (44);
\draw[m,dashed](5)to []node[midway,left,yshift=-2pt,xshift=2pt] {\footnotesize{$\lambda$}} (55);
\draw[m](4) to[bend left]node[midway,yshift=0pt,xshift=7pt] {\footnotesize{$1$}}  (444);
\draw[m, dashed](5) to[bend right] node[midway,right,yshift=-3pt,xshift=-17pt] {\footnotesize{$A$}} (555);
\end{tikzpicture}
\caption{}\label{fig:1.2}
\end{subfigure}}
\scalebox{0.8}{
\begin{subfigure}{0.33\textwidth}
\centering
\begin{tikzpicture}
[n/.style={circle,draw=black,scale=0.7,inner sep=2pt},
nn/.style={circle,,scale=0.7,inner sep=2pt, draw=black, densely dotted},
m/.style={draw=black}
]
\node[n] (1) at (-1,0.5){};
\node[n] (2) at (1,0.5){}; 
\node[n] (3) at (0,-0.5){$x$};
\node[n] (5) at (-1,-1.5){$x_1$};
\node[n] (4) at (1,-1.5){$x_2$};
\node[n] (555) at (-2,-2.5){};
\node[n] (55) at (-.5,-2.5){};
\node[n] (44) at (.5,-2.5){};
\node[n] (444) at (2,-2.5){};
\node[nn] (d5) at (-1.3,-2.5){};
\node[nn] (d4) at (1.3,-2.5){};

\draw[m](1) to[bend left=50]node[midway,yshift=7pt,xshift=-1pt] {$C_i$} (2);
\draw[m](3) to[]node[midway,yshift=7pt,xshift=-1pt] {\footnotesize{$2$}}(4);
\draw[m](3) to[]node[midway,yshift=7pt,xshift=0pt]{\footnotesize{$1$}} (5);
\draw[m](1)--(3)node[midway,above,yshift=0pt,xshift=0pt] {\footnotesize{$ $}};
\draw[m](2)--(3)node[midway,above,yshift=-1pt,xshift=0pt] {\footnotesize{$ $}};
\draw[m](4)to[]  node[midway,left,yshift=0pt,xshift=11pt] {\footnotesize{$1$}} (44);
\draw[m](5)to []node[midway,right,yshift=0pt,xshift=-12pt] {\footnotesize{$A$}} (55);
\draw[m](5) to[bend right]node[midway,left,yshift=0pt,xshift=1pt] {\footnotesize{$B$}}  (555);
\draw[m](4) to[bend left]node[midway,right,yshift=0pt,xshift=0pt] {\footnotesize{$B$}}  (444);
\draw[m](44) to node[midway,yshift=-5pt,xshift=0pt] {\footnotesize{$B$}}  (55);
\draw[m, dashed](5) to node[midway,right,yshift=-3pt,xshift=-15pt] {\footnotesize{$A$}} (d5);
\draw[m, dashed](4) to node[midway,right,yshift=-3pt,xshift=3pt] {\footnotesize{$A$}} (d4);
\end{tikzpicture}
\caption{}\label{fig:1.3}
\end{subfigure}
}
\caption{Star edge coloring of cycle $D_x$ of length $3$, $4$, and $5$}{\label{fig2}}
\end{figure}
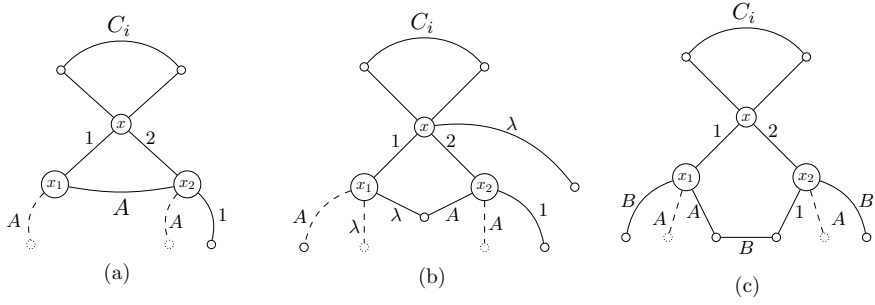
\begin{figure}[H]
\centering
\scalebox{0.8}{
\begin{subfigure}{0.33\textwidth}
\centering
\begin{tikzpicture}
[n/.style={circle,draw=black,scale=0.7,inner sep=2pt},
nn/.style={circle,,scale=0.7,inner sep=2pt, draw=black, densely dotted},
m/.style={draw=black}
]
\node[n] (1) at (-1,0.4){};
\node[n] (2) at (1,0.4){}; 
\node[n] (3) at (0,-0.5){$x$};
\node[n] (5) at (-1.5,-1){$x_1$};
\node[n] (4) at (1.5,-1){$x_2$};
\node[n] (55) at (-1,-2){};
\node[n] (44) at (1,-2){};
\node[n] (444) at (2.8,-2){};
\node[n] (c1) at (0,-3){};
\node[nn] (d5) at (-2,-2){};
\node[nn] (d4) at (2,-2){};

\draw[m](1) to[bend left=50]node[midway,yshift=7pt,xshift=-1pt] {$C_i$} (2);
\draw[m](3) to[]node[midway,yshift=7pt,xshift=-1pt] {\footnotesize{$2$}}(4);
\draw[m](3) to[]node[midway,yshift=7pt,xshift=0pt]{\footnotesize{$1$}} (5);
\draw[m](1)--(3)node[midway,above,yshift=0pt,xshift=0pt] {\footnotesize{$ $}};
\draw[m](2)--(3)node[midway,above,yshift=-1pt,xshift=0pt] {\footnotesize{$ $}};
\draw[m](4)to[]  node[midway,right,yshift=0pt,xshift=-1pt] {\footnotesize{$A$}} (44);
\draw[m](5)to []node[midway,left,yshift=0pt,xshift=0pt] {\footnotesize{$A$}} (55);
\draw[m](44) to node[midway,right,yshift=0pt,xshift=0pt] {\footnotesize{$1$}}(c1);
\draw[m](55) to node[midway,left,yshift=0pt,xshift=0pt] {\footnotesize{$2$}}(c1);
\draw[m](4) to[bend left]node[midway,yshift=0pt,xshift=7pt] {\footnotesize{$1$}}  (444);
\draw[m,dashed](4) to[bend left]node[midway,yshift=0pt,xshift=7pt] {\footnotesize{$A$}}  (d4);
\draw[m,dashed](5) to[bend right]node[midway,yshift=0pt,xshift=-5pt] {\footnotesize{$A$}}  (d5);
\end{tikzpicture}
\caption{}\label{fig:3.1}
\end{subfigure}}
\scalebox{0.8}{
\begin{subfigure}{0.33\textwidth}
\centering
\begin{tikzpicture}
[n/.style={circle,draw=black,scale=0.7,inner sep=2pt},
nn/.style={circle,,scale=0.7,inner sep=2pt, draw=black, densely dotted},
m/.style={draw=black}
]
\node[n] (1) at (-1,0.4){};
\node[n] (2) at (1,0.4){}; 
\node[n] (3) at (0,-0.5){$x$};
\node[n] (5) at (-1.5,-1){$x_1$};
\node[n] (4) at (1.5,-1){$x_2$};
\node[n] (55) at (-1,-2){ };
\node[n] (44) at (1,-2){ };
\node[n] (444) at (2.7,-2){ };
\node[n] (c1) at (.6,-3){};
\node[n] (c2) at (-.6,-3){};
\node[nn] (d5) at (-2,-2){};
\node[nn] (d4) at (2,-2){};

\draw[m](1) to[bend left=50]node[midway,yshift=7pt,xshift=-1pt] {$C_i$} (2);
\draw[m](3) to[]node[midway,yshift=7pt,xshift=-1pt] {\footnotesize{$2$}}(4);
\draw[m](3) to[]node[midway,yshift=7pt,xshift=0pt]{\footnotesize{$1$}} (5);
\draw[m](1)--(3)node[midway,above,yshift=0pt,xshift=0pt] {\footnotesize{$ $}};
\draw[m](2)--(3)node[midway,above,yshift=-1pt,xshift=0pt] {\footnotesize{$ $}};
\draw[m](4)to[]  node[midway,right,yshift=0pt,xshift=0pt] {\footnotesize{$1$}} (44);
\draw[m](5)to []node[midway,left,yshift=0pt,xshift=0pt] {\footnotesize{$A$}} (55);
\draw[m](44) to node[midway,right,yshift=0pt,xshift=0pt] {\footnotesize{$A$}}(c1);
\draw[m](55) to node[midway,left,yshift=0pt,xshift=0pt] {\footnotesize{$1$}}(c2);
\draw[m](4) to[bend left]node[midway,yshift=0pt,xshift=7pt] {\footnotesize{$A$}}  (444);
\draw[m](c2) to  node[midway,yshift=-7pt,xshift=0pt] {\footnotesize{$2$}}(c1);
\draw[m,dashed](4) to[bend left]node[midway,yshift=0pt,xshift=7pt] {\footnotesize{$1$}}  (d4);
\draw[m,dashed](5) to[bend right]node[midway,yshift=0pt,xshift=-5pt] {\footnotesize{$A$}}  (d5);
\end{tikzpicture}
\caption{}\label{fig:3.2}
\end{subfigure}}
\scalebox{0.8}{
\begin{subfigure}{0.33\textwidth}
\centering
\begin{tikzpicture}
[n/.style={circle,draw=black,scale=0.7,inner sep=2pt},
nn/.style={circle,,scale=0.7,inner sep=2pt, draw=black, densely dotted},
m/.style={draw=black}
]
\node[n] (1) at (-1,0.6){};
\node[n] (2) at (1,0.6){}; 
\node[n] (3) at (0,0){$x$};
\node[n] (5) at (-1.6,-0.3){$x_1$};
\node[n] (4) at (1.6,-0.3){$x_2$};
\node[n] (55) at (-1.2,-1.1){};
\node[n] (44) at (1.2,-1.1){};
\node[n] (444) at (2.7,-1.1){};
\node[n] (c1) at (.9,-1.9){};
\node[n] (c2) at (-.9,-1.9){};
\node[n] (666) at (-.6,-2.7){};
\node[n] (777) at (.6,-2.7){};
\node[n] (888) at (.3,-3.5){};
\node[n] (999) at (-.3,-3.5){};
\node[nn] (d5) at (-2,-1.2){};
\node[nn] (d4) at (2,-1.2){};

\draw[m](1) to[bend left=70]node[midway,yshift=7pt,xshift=-1pt] {$C_i$} (2);
\draw[m](3) to[]node[midway,yshift=5pt,xshift=-1pt] {\footnotesize{$2$}}(4);
\draw[m](3) to[]node[midway,yshift=5pt,xshift=0pt]{\footnotesize{$1$}} (5);
\draw[m](1)--(3)node[midway,above,yshift=0pt,xshift=0pt] {\footnotesize{$ $}};
\draw[m](2)--(3)node[midway,above,yshift=-1pt,xshift=0pt] {\footnotesize{$ $}};
\draw[m](4)to[]  node[midway,right,yshift=0pt,xshift=0pt] {\footnotesize{$1$}} (44);
\draw[m](5)to []node[midway,left,yshift=0pt,xshift=0pt] {\footnotesize{$A$}} (55);
\draw[m](44) to node[midway,right,yshift=0pt,xshift=0pt] {\footnotesize{$A$}}(c1);
\draw[m](55) to node[midway,left,yshift=0pt,xshift=0pt] {\footnotesize{$2$}}(c2);
\draw[m](4) to[bend left]node[midway,yshift=0pt,xshift=7pt] 
{\footnotesize{$A$}}  (444);
\draw[m](c1) to[]node[midway,yshift=0pt,xshift=7pt] {\footnotesize{$2$}}  (777);
\draw[m](888) to[]node[midway,yshift=-2pt,xshift=7pt] {\footnotesize{$1$}}  (777);
\draw[m](c2) to[]node[midway,left,yshift=-1pt,xshift=0pt] {\footnotesize{$1$}}  (666);
\draw[m](666) to[]node[midway,left,yshift=-2pt,xshift=0pt] 
{\footnotesize{2}}  (999);
\draw[m](888) to[]node[midway,yshift=-6pt,xshift=0pt] {\footnotesize{$A$}}  (999);
\draw[m,dashed](4) to[bend left]node[midway,yshift=0pt,xshift=7pt] {\footnotesize{$1$}}  (d4);
\draw[m,dashed](5) to[bend right]node[midway,yshift=0pt,xshift=-5pt] {\footnotesize{$A$}}  (d5);
\end{tikzpicture}
\caption{}\label{fig:3.3}
\end{subfigure}}
\caption{Star edge coloring of cycle $D_x$ of length at least $6$.}\label{fig3}
\end{figure}
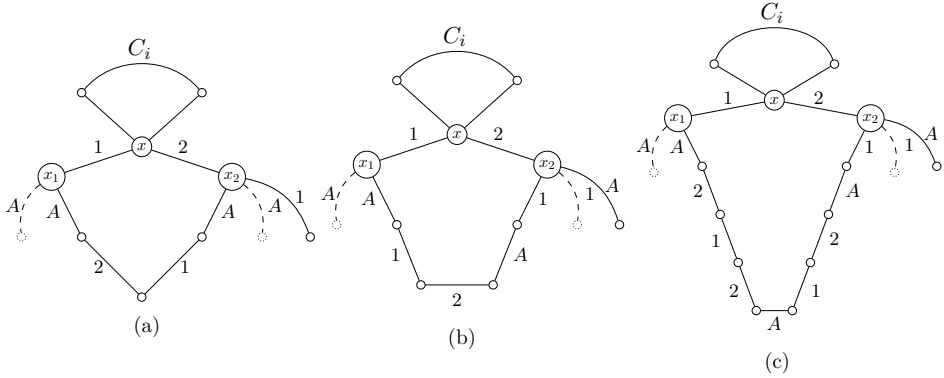
\section{Tightness of the bound}\label{tightness}
In this section, we prove that the given bound in Theorem~\ref{thm:main-result1} is tight and there are infinite family of Cactus graphs that achive the bound $\left\lfloor\frac{3\Delta}{2}\right\rfloor+1$.
First we need to see some facts about the star edge coloring of trees.

\begin{lem}\label{facts}
Let $\Delta$ be an odd positive integer and $T_v$ be a $\Delta$-semiregular rooted tree of height  two. In every   star edge coloring $c$  of $T_v$ with color set  $\mathcal{C}=\{1,\ldots,\left\lfloor\frac{3\Delta}{2}\right\rfloor\}$ and   for every two neighbours  $x$ and $y$  of root~$v$, we have
 \begin{itemize}
\item[\rm{(a)}]
$|C(x)\cap C(y)|\geq \frac{\Delta+1}{2}$, $C^\prime(v)=\mathcal{C}\setminus C(v)\subset C(x)$.
\item[\rm{(b)}]
  Every color of $C(v)$ is used  for  $\frac{\Delta-1}{2}$ non-incident edges to  $v$.
\item[\rm{(c)}]
$c(vx)\in C(y)$ or $c(vy)\in C(x)$.
\end{itemize}
\end{lem}
\begin{proof}
{
It is clear that
 \[|C(x)\cap C(y)|\geq 2\Delta-\left\lfloor\frac{3\Delta}{2}\right\rfloor=\frac{\Delta+1}{2}.\] 
 Let $n_i$ be  the numbers of edges colored with $i$  in $T_v\setminus \{v\}$. Obviously, 
 $C(x)\cap C^\prime(v)\leq\frac{\Delta-1}{2}$.
 Moreover, for every   two neighbours $x$ and $y$ of $v$, if color $c(vx)$
is used for some  incident edges to $y$, then color $c(vy)$ is not used for the  incident edges to $x$, and vice versa. Thus,  for each color $i\in C(v)$,  we have $n_i\leq  \frac{\Delta-1}{2}$. Also,  each color of $C^\prime(v)$ can be used for the incident edges of each neighbour of $v$.   Since  $T_v$ has $\Delta(\Delta-1)$ leaves,  we must have
\[\Delta(\Delta-1)=\sum_{i\in C(v)}{n_i}+\sum_{i\in C^\prime(v)}{n_i}\leq \sum_{i\in C(v)}{\frac{\Delta-1}{2}}+\sum_{i\in C^\prime(v)}{\Delta}= \Delta(\Delta-1), \] which implies that  if $i \in C^{\prime}(v)$, then $n_i=\Delta$ and proves (a). Moreover, if $ i \in C(v)$, then $n_i=\frac{\Delta-1}{2}$, that proves (b).

To prove (c), by contrary, suppose that  $c(vx)\not\in C(y)$ and $c(vy)\not\in C(x)$. Then, colors $c(vx)$ and $c(vy)$ can be used for at most
 $(\Delta-2)-\frac{\Delta-1}{2}<\frac{\Delta-1}{2}$ non-incident edges to $v$, that  contradicts~(b). 
}
\end{proof}
\begin{thm}\label{thm:main-result}
For  every odd positive integer $\Delta\geq 3$, there exists Cactus $G$ with maximum degree~$\Delta$, in which $\chi^\prime_s(G)=\left\lfloor\frac{3\Delta}{2}\right\rfloor+1$.
\end{thm}
\begin{proof}
{
For every positive odd integer $\Delta$, we construct a Cactus that achives the bound. 
 For this purpose, let $T_v$ be a $\Delta$-semiregular tree of height three with root $v$, and  $x$ and $y$ are two neighbours of $v$. For every vertex $u\neq v$ in $T_v$, we denote the neighbours of $u$ by $\{u_0,u_1,\ldots,u_{\Delta-1}\}$, where $u_0$ is the parent of $u$.
 
 We construct Cactus  $G_T$ from $T_v$  by adding edge $xy$ and removing vertices $x^{\prime}=x_{\Delta-1}$ and $y^{\prime}=y_{\Delta-1}$.
 Obviously, if $G_T$ admits a star edge coloring $\phi$ with $\left\lfloor\frac{3\Delta}{2}\right\rfloor$ colors, then $T_v$ has  star edge coloring $c$ with  $\left\lfloor\frac{3\Delta}{2}\right\rfloor$ colors, where
 \[ c(e)=\begin{cases}
\phi(e)~~&\text{if}~e\in G_T,
\\
 \phi(xy)~~& \text{if}~e=xx^{\prime} \ or \ e= yy^{\prime},
 \\
 \phi(xx_{i-1})~~& \text{if}~e=y^{\prime}y^{\prime}_i, \ \ \  1 \leq  i \leq \Delta-1,
 \\
  \phi(yy_{i-1})~~& \text{if}~e=x^{\prime}x^{\prime}_i, \ \ \ 1 \leq  i \leq \Delta-1.
\end{cases}
 \]
By  Lemma~\ref{facts}(c), $\phi(vx)\not\in C(y)$ or $\phi(vy)\not\in C(x)$, hence $\phi(xy)\not\in C(v)$.
 Moreover, without loss of generality, assume that $\phi(vy)\in C(x)$. 
  Lemma~\ref{facts}(a),  implies that $C(x) \cap C(y)=\frac{\Delta+1}{2}$. 
 Therefore, color $\phi(xy)$ can be used for at most  $(\Delta-1)-\frac{\Delta+1}{2}<\frac{\Delta-1}{2}$ incident edges to the neighbours of $x$, except  $v$.  That contradicts   Lemma~\ref{facts}(b) for $\Delta$-semiregular subtree $T_x$ in $T_v$. Therefore, 
 $\chi^\prime_s(G_T)\geq \left\lfloor\frac{3\Delta}{2}\right\rfloor+1$. Hence,  by Theorem~\ref{thm:main-result1} we have $\chi^\prime_s(G_T)=\left\lfloor\frac{3\Delta}{2}\right\rfloor+1$.
 }
\end{proof}
\begin{thm}\label{lem:6-regular}
There exist  Cactus $G$ with maximum degree 
$6$, where $\chi_{s}^{\prime}(G)=10$
\end{thm}
\begin{proof}

Let $G$ be the Cactus shown in Figure \ref{fig:10}. We show that $G$ has no star edge coloring with color set $\mathcal{C}=\{1, \ldots,9\}$. Assume that $C:=x_1x_2x_3$ is the cycle of length three in $G$ colored by $1,2$ and $3$.
In every star edge coloring $E_1(C)\cup E_2(C)$ with colors $\mathcal{C}$, there is vertex $x_i$ that the color set of edges in $E_2(C,x_i)$ is a subset of $\{4,5,6,7,8\}$ and the colors of at least  two edges in $E_2(C,x_i)$ is used at least two times in the color set of edges in $E_2(C)$.
Without loss of generality, assume that $i=3$ and $C(x_3)=\{1,3,4,5,6,7\}$. Let $D:=x_3y_1y_2y_3$ and $c(x_3y_1)=4$, $c(x_3y_3)=5$.
 The common  possible colors that we can use for edges incident to  $y_1$ and $y_3$ are $\{2,8,9\}= \mathcal{C} \setminus C(x_3)$, we color three edges in  $E_2(D,y_1)$ and $E_2(D,y_3)$ with these colors. On the other hand color $3$ and $1$ are usable for coloring an edge in $E_2(D,y_1)$ and   $E_2(D,y_3)$, respectively.  If we set $c(y_1y_2)=A \in \{2,8,9\}$ and $c(y_2y_3)=B\in \{2,8,9\} $ or we set $c(y_1y_2)=A \in \{2,8,9\}$ and $c(y_2y_3)=4$, then we have bi-colored path of length four in $G$. Otherwise, we must have $c(y_1y_2)=A \in \{2,8,9\} $ and $c(y_2y_3)=1$ (or $c(y_1y_2)=3$ and $c(y_2y_3)=A$). 

Now we show that this is also impossible.  
 For this purpose, let $A=8$ and color the edges in $E_3(D,y_1z_i)$, $ 1\leq i \leq 4$, with colors $\{1, \ldots, 9\} \setminus \{8\}$ such that there is no bi-colored path of length four.  First,  color one edge in $E_2(D,y_2)$ with color $7$. According to colors of $E_3(D,y_1z_1)$, we can color two edge of $E_2(D,y_2)$ with $2$ and $9$.
On the other hand, the only possible edges that can be colored with color $1$ is one edge in $E_3(D, y_3z_{12})$. This implies that there is no possible color to color edge $y_2z_8$ with a color in the color set of edges in $E_2(D,y_3)\cup E_3(D,y_3z_9)\cup E_3(D,y_3z_10)$. By a similar discussion it can be shown that if we set
 $c(y_1y_2)=3$ and $c(y_2y_3)=1$
 or 
 $c(y_1y_2)=6$ and $c(y_2y_3)=1$ 
 or
$c(y_1y_2)=6$ and $c(y_2y_3)=A$, 
 then coloring of $G$ with $9$ colors is impossible.
 \end{proof}

\begin{con} For every even integer $\Delta\geq 6$, there exists $\Delta$-semiregular Cactus  $G$ with  star chromatic index $\left\lfloor\frac{3\Delta}{2}\right\rfloor+1$.
\end{con}
\begin{figure}[H]
\centering
\scalebox{0.8}{
\begin{tikzpicture}
[n/.style={circle,draw=black,scale=1.1,inner sep=1pt},
nn/.style={circle,draw=yellow,scale=1.15,inner sep=2pt},
nnn/.style={circle,draw=white,scale=1,inner sep=1pt},
m/.style={draw=black}]

		\node [n] (0) at (-1.25, 1.5) {$x_1$};
		\node [n] (1) at (1.25, 1.5) {$x_2$};
		\node [n] (2) at (0, -0.75) {$x_3$};
		\node [n] (3) at (-1.25, -2.75) {$y_1$};
		\node [n] (4) at (1.25, -2.75) {$y_3$};
		\node [n] (5) at (0, -4.75) {$y_2$};
		\node [n] (6) at (-2.5, 2.25) {$ $};
		\node [n] (7) at (-2, 2.75) {$ $};
		\node [n] (8) at (-1.25, 3) {$ $};
		\node [n] (9) at (-0.5, 3) {$ $};
		\node [n] (10) at (0.5, 3) {$ $};
		\node [n] (11) at (1.25, 3) {$ $};
		\node [n] (12) at (2, 2.75) {$ $};
		\node [n] (13) at (2.5, 2.25) {$ $};
		\node [n] (14) at (1.75, -0.75) {$ $};
		\node [n] (15) at (-1.75, -0.75) {$ $};
		\node [n] (16) at (3, -1) {$z_9$};
		\node [n] (17) at (3.25, -2) {$z_{10}$};
		\node [n] (18) at (3.25, -3) {$z_{11}$};
		\node [n] (19) at (3, -4) {$z_{12}$};
		\node [n] (20) at (-1.25, -6.25) {$z_5$};
		\node [n] (21) at (-0.5, -6.5) {$z_6$};
		\node [n] (22) at (0.5, -6.5) {$z_7$};
		\node [n] (23) at (1.25, -6.25) {$z_8$};
		\node [n] (24) at (-3, -1) {$z_1$};
		\node [n] (25) at (-3.25, -2) {$z_2$};
		\node [n] (26) at (-3.25, -3) {$z_3$};
		\node [n] (27) at (-2.75, -4) {$z_4$};
		\node [n] (28) at (-3, 1) {$ $};
		\node [n] (29) at (-3.75, 1) {$ $};
		\node [n] (30) at (-4.5, 1) {$ $};
		\node [n] (31) at (-5.25, 1) {$ $};
		\node [n] (32) at (-2.25, 1) {$ $};
		\node [n] (33) at (-5.6, 0.7) {$ $};
		\node [n] (34) at (-6, 0.25) {$ $};
		\node [n] (35) at (-6, -0.5) {$ $};
		\node [n] (36) at (-6, -1.25) {$ $};
		\node [n] (37) at (-6, -2) {$ $};
		\node [n] (38) at (-6, -3) {$ $};
		\node [n] (39) at (-6, -3.75) {$ $};
		\node [n] (40) at (-6, -4.5) {$ $};
		\node [n] (41) at (-6, -5.25) {$ $};
		\node [n] (42) at (-5.6, -5.7) {$ $};
		\node [n] (43) at (-5.25, -6) {$ $};
		\node [n] (44) at (-4.5, -6) {$ $};
		\node [n] (45) at (-3.75, -6) {$ $};
		\node [n] (46) at (-3, -6) {$ $};
		\node [n] (47) at (-2.25, -6) {$ $};
		\node [n] (48) at (2.25, 1) {$ $};
		\node [n] (49) at (3, 1) {$ $};
		\node [n] (50) at (3.75, 1) {$ $};
		\node [n] (51) at (4.5, 1) {$ $};
		\node [n] (52) at (5.25, 1) {$ $};
		\node [n] (53) at (5.6, 0.7) {$ $};
		\node [n] (54) at (6, 0.25) {$ $};
		\node [n] (55) at (6, -0.5) {$ $};
		\node [n] (56) at (6, -1.25) {$ $};
		\node [n] (57) at (6, -2) {$ $};
		\node [n] (58) at (6, -3) {$ $};
		\node [n] (59) at (6, -3.75) {$ $};
		\node [n] (60) at (6, -4.5) {$ $};
		\node [n] (61) at (6, -5.25) {$ $};
		\node [n] (62) at (5.6, -5.7) {$ $};
		\node [n] (63) at (5.25, -6) {$ $};
		\node [n] (64) at (4.5, -6) {$ $};
		\node [n] (65) at (3.75, -6) {$ $};
		\node [n] (66) at (3, -6) {$ $};
		\node [n] (67) at (2.25, -6) {$ $};

\draw [m](9)to[]node[midway,yshift=4pt,xshift=7pt]{$5$}(0);
\draw [m](1)to[]node[midway,yshift=-6pt,xshift=0pt]{$2$}(0);
\draw [m](0)to[]node[midway,yshift=0pt,xshift=6pt]{$3$}(2);
\draw [m](2)to[]node[midway,yshift=0pt,xshift=-5pt]{$1$}(1);

\draw [m](2)to[]node[midway,yshift=-5pt,xshift=+4pt]{$4$}(3);
\draw [m](2)to[]node[midway,yshift=6pt,xshift=0pt]{$7$}(14);
\draw [m](2)to[]node[midway,yshift=6pt,xshift=0pt]{$6$}(15);
\draw [m](2)to[]node[midway,yshift=-5pt,xshift=-4pt]{$5$}(4);

\draw [m](1)to[]node[midway,yshift=4pt,xshift=-8pt]{$6$}(10);
\draw [m](1)to[]node[midway,yshift=4pt,xshift=-5pt]{$4$}(11);
\draw [m](1)to[]node[midway,yshift=6pt,xshift=-3pt]{$9$}(12);
\draw [m](1)to[]node[midway,yshift=7pt,xshift=0pt]{$8$}(13);
		
\draw [m](0)to[]node[midway,yshift=4pt,xshift=5pt]{$7$}(8);
\draw [m](0)to[]node[midway,yshift=6pt,xshift=2pt]{$8$}(7);
\draw [m](0)to[]node[midway,yshift=6pt,xshift=0pt]{$9$}(6);

\draw [m](3)to[]node[midway,yshift=6pt,xshift=4pt]{$2$}(24);
\draw [m](3)to[]node[midway,yshift=6pt,xshift=0pt]{$3$}(25);
\draw [m](3)to[]node[midway,yshift=10pt,xshift=4pt]{$8$}(5);
\draw [m](3)to[]node[midway,yshift=5pt,xshift=-6pt]{$6$}(26);
\draw [m](3)to[]node[midway,yshift=2pt,xshift=-7pt]{$9$}(27);

\draw [m](5)to[]node[midway,yshift=5pt,xshift=-4pt]{\color{red}{$2$}}(20);
\draw [m](5)to[]node[midway,yshift=0pt,xshift=-3pt]{\color{red}{$7$}}(21);
\draw [m](5)to[]node[midway,yshift=0pt,xshift=-4pt]{\color{red}{$9$}}(22);
\draw [m](5)to[]node[midway,yshift=-3pt,xshift=-3pt]{\color{red}{$?$}}(23);
\draw [m](5)to[]node[midway,yshift=10pt,xshift=-4pt]{$1$}(4);
		
\draw [m](4)to[]node[midway,yshift=3pt,xshift=8pt]{$9$}(19);
\draw [m](4)to[]node[midway,yshift=5pt,xshift=6pt]{$8$}(18);
\draw [m](4)to[]node[midway,yshift=10pt,xshift=4pt]{$4$}(17);
\draw [m](4)to[]node[midway,yshift=11pt,xshift=3pt]{$2$}(16);

\draw [m](32)to[]node[midway,yshift=9pt,xshift=8pt]{$1$}(24);
\draw [m](28)to[]node[midway,yshift=10pt,xshift=5pt]{$5$}(24);
\draw [m](29)to[]node[midway,yshift=10pt,xshift=4pt]{$7$}(24);
\draw [m](30)to[]node[midway,yshift=11pt,xshift=0pt]{$6$}(24);
\draw [m](31)to[]node[midway,yshift=11pt,xshift=-4pt]{\color{blue}{$4$}}(24);
		
\draw [m](33)to[]node[midway,yshift=12pt,xshift=-3pt]{$1$}(25);
\draw [m](34)to[]node[midway,yshift=14pt,xshift=-9pt]{$5$}(25);
\draw[m](35)to[]node[midway,yshift=14pt,xshift=-13pt]{$7$}(25);
\draw[m](36)to[]node[midway,yshift=11pt,xshift=-16pt]{$2$}(25);
\draw[m](37)to[]node[midway,yshift=8pt,xshift=-17pt]{$9$}(25);
		
\draw [m](26)to[]node[midway,yshift=5pt,xshift=-6pt]{$1$}(38);
\draw [m](26)to[]node[midway,yshift=3pt,xshift=-10pt]{$5$}(39);
\draw [m](26)to[]node[midway,yshift=0pt,xshift=-7pt]{$7$}(40);
\draw [m](26)to[]node[midway,yshift=-1pt,xshift=-9pt]{$3$}(41);
\draw[m](26)to[]node[midway,yshift=-3pt,xshift=-9pt]{\color{blue}{$8$}}(42);
		
\draw [m](27)to[]node[midway,yshift=8pt,xshift=0pt]{$1$}(43);
\draw [m](27)to[]node[midway,yshift=-2pt,xshift=-6pt]{$5$}(44);
\draw [m](27)to[]node[midway,yshift=-8pt,xshift=-7pt]{$7$}(45);
\draw [m](27)to[]node[midway,yshift=-10pt,xshift=-5pt]{$2$}(46);
\draw [m](27)to[]node[midway,yshift=-10pt,xshift=-5pt]{$6$}(47);
		
\draw [m](19)to[]node[midway,yshift=-5pt,xshift=4pt]{\color{blue}{$5$}}(67);
\draw [m](19)to[]node[midway,yshift=-5pt,xshift=4pt]{$4$}(66);
\draw [m](19)to[]node[midway,yshift=-5pt,xshift=7pt]{$7$}(65);
\draw [m](19)to[]node[midway,yshift=-4pt,xshift=8pt]{$6$}(64);
\draw [m](19)to[]node[midway,yshift=5pt,xshift=3pt]{$3$}(63);
		
\draw [m](18)to[]node[midway,yshift=-2pt,xshift=9pt]{$2$}(62);
\draw [m](18)to[]node[midway,yshift=0pt,xshift=9pt]{$9$}(61);
\draw [m](18)to[]node[midway,yshift=2pt,xshift=6pt]{$7$}(60);
\draw [m](18)to[]node[midway,yshift=5pt,xshift=6pt]{$3$}(58);
\draw [m](18)to[]node[midway,yshift=4pt,xshift=7pt]{$6$}(59);
		
\draw[m](17)to[]node[midway,yshift=6pt,xshift=16pt]{\color{blue}{$1$}}(57);
\draw [m](17)to[]node[midway,yshift=9pt,xshift=15pt]{$8$}(56);
\draw [m](17)to[]node[midway,yshift=15pt,xshift=16pt]{$7$}(55);
\draw [m](17)to[]node[midway,yshift=16pt,xshift=13pt]{$6$}(54);
\draw [m](17)to[]node[midway,yshift=10pt,xshift=4pt]{$3$}(53);
		
\draw [m](16)to[]node[midway,yshift=11pt,xshift=3pt]{$4$}(52);
\draw [m](16)to[]node[midway,yshift=9pt,xshift=-8pt]{$3$}(48);
\draw [m](16)to[]node[midway,yshift=10pt,xshift=-4pt]{$6$}(49);
\draw [m](16)to[]node[midway,yshift=10pt,xshift=-1pt]{$7$}(50);
\draw [m](16)to[]node[midway,yshift=12pt,xshift=3pt]{$9$}(51);
\end{tikzpicture}
}
\caption{A Cactus $G$ with $\Delta=6$ and $\chi_s(G)=10$.}\label{fig:10}
\end{figure}
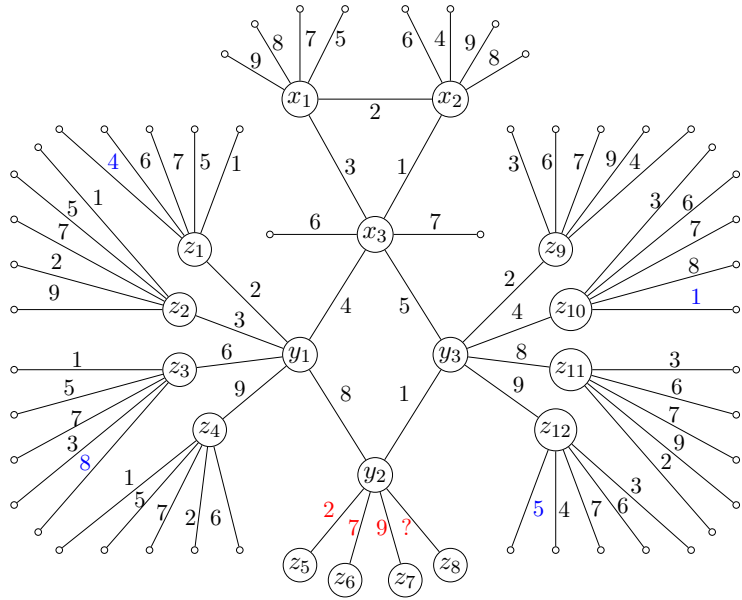


\end{document}